\documentclass[a4paper]{article}

\usepackage[all]{xy}\usepackage[latin1]{inputenc}        
\usepackage[dvips]{graphics,graphicx}
\usepackage{amsfonts,amssymb,amsmath,color,mathrsfs, amstext,bm}
\usepackage{amsbsy, amsopn, amscd, amsxtra, amsthm,authblk, enumerate}
\usepackage{upref}
\usepackage[colorlinks,
            linkcolor=red,
            anchorcolor=red,
            citecolor=red
            ]{hyperref}

\usepackage{geometry}
\usepackage[displaymath]{lineno}
\usepackage{float}
\usepackage{tikz-cd}
\usepackage[ruled]{algorithm2e}
\usepackage{yhmath}

\usepackage[normalem]{ulem}

\allowdisplaybreaks

\numberwithin{equation}{section}

\def\e{\text{e}}

\def\E{\mathbb{E}}
\def\R{\mathbb{R}}

\newtheorem{theorem}{Theorem}[section]
\newtheorem{lemma}{Lemma}[section]
\newtheorem{assumption}{Assumption}[section]
\newtheorem{proposition}{Proposition}[section]
\newtheorem{remark}{Remark}[section]
\newtheorem{corollary}{Corollary}[section]

\begin{document}
\title{Propagation of chaos and approximation error of random batch particle system in the mean field regime}

\author[a,b]{Lei Li\thanks{leili2010@sjtu.edu.cn}}
\author[a]{Yuelin Wang\thanks{Corresponding author: sjtu$\_$wyl@sjtu.edu.cn}}
\author[a,b]{Shi Jin\thanks{shijin-m@sjtu.edu.cn}}
\affil[a]{School of Mathematical Sciences, Shanghai Jiao Tong University, Shanghai 200240, China }
\affil[b]{Institute of Natural Sciences and MOE-LSC, Shanghai Jiao Tong University, Shanghai 200240, China }

\date{}
\maketitle
\begin{abstract}
    The random batch method [J. Comput. Phys. 400 (2020) 108877] is not only an efficient algorithm for simulation of classical $N$-particle systems and their mean-field limit, but also a new model for interacting particle system that could be more physical in some applications.
In this work, we establish the propagation of chaos for the random batch particle system and at the same time obtain its sharp approximation error to the classical mean field limit of $N$-particle systems. The proof leverages the BBGKY hierarchy and achieves a sharp bound both in the particle number $N$ and the time step $\tau$. In particular, by introducing a coupling of the division of the random batches to resolve the $N$-dependence, we derive an $\mathcal{O}(k^2/N^2 + k\tau^2)$ bound on the $k$-particle relative entropy between the law of the system and the tensorized law of the mean-field limit. 
This result provides a useful understanding of the convergence properties of the random batch system in the mean field regime.

\textbf{Keywords:} Interacting particle systems, Random Batch Method, mean-field limit, BBGKY, propagation of chaos.
\end{abstract}

\section{Introduction}
Interacting particle systems are prevalent in numerous significant problems across physical, social, and biological sciences. Examples include molecular dynamics \cite{liang2022superscalability,jin2021randomEwald}, swarming \cite{toner1998flocks}, chemotaxis \cite{horstmann20031970,bertozzi2012characterization}, flocking \cite{ha2009simple,Ha2008particle,cucker2007emergent,motsch2011new}, synchronization \cite{kuramoto1975self,buck1966biology}, consensus \cite{motsch2014heterophilious}, and random vortex models \cite{majda2002vorticity}. 
In this work, we are interested in the following general first-order system of $N$ particles with constant diffusion $\sigma$:
\begin{equation}\label{eq:particle}
    dX_i = b_0(X_i)dt + \frac{1}{N-1}\sum\limits_{j=1,\,j\ne i}^N b(X_i - X_j)dt + \sqrt{2\sigma} dW_i,\quad i=1,\cdots, N,
\end{equation}
with $X_i(0) = X_i^0$ drawn independently from the same initial distribution $\mu_0$.  
Here, $X_i = X_i(t)$, $W_i=W_i(t)$, where we omit the explicit dependence on $t$ for notational simplicity. Also, $\{X_i\in\R^d\}$ are the labels for particles; $b_0:\R^d\to\R^d$ denotes the drift force and $b:\R^d\to\R^d$ denotes the interaction kernel, and $\{W_i\}_{i=1}^N$ are $N$ independent $d$-dimensional Brownian motions.

It is well-known that simulating the $N$-particle system can be highly challenging, as $N$ is often very large in practical applications. To address this issue, theoretically, one classic approach is to consider the mean-field limit, which provides a simplified yet effective description of the system. Unlike the $N$-particle dynamics described in \eqref{eq:particle}, the mean-field limit focuses on a one-particle effective dynamics when $N\to \infty$. For \eqref{eq:particle}, the limiting behavior of any particle is described by the McKean SDE
\begin{equation}\label{eq:mckean0}
    d\bar{X}_1 = b_0(\bar{X}_1)dt + b*\bar{\mu}_t (\bar{X}_1)dt + \sqrt{2\sigma} dW_1,
\end{equation}
where $\bar{\mu}_t$ denotes the law of $\bar{X}_1$ and $\bar{X}_1(0)$ is sampled independently from $\mu_0$. 

It is well-established that the large $N$ limit of particle systems can be mathematically formalized through the concept of propagation of chaos, first introduced in \cite{kac1956foundations}. 
Recently, entropy methods have gained a lot of attraction for quantitative propagation of chaos.
Jabin and Wang employed combinatorial analysis to derive a series of estimates for relative entropy, as seen in \cite{jabin2016mean,jabin2017mean,jabin2018quantitative} for instance. A significant improvement in 2023 came from Lacker \cite{lacker2023hierarchies}, where the BBGKY hierarchies is utilized to establish optimal local convergence rates of $N$ for relative entropy. Following this, several works extended this technique to various cases, such as \cite{hao2024strong,grass2024sharp,cattiaux2024entropy}. Most recently, Bresch et al. \cite{bresch2024duality} developed a novel duality-based method for $L^2$ kernels, further advancing the field. Some reviews see \cite{golse2016dynamics,jabin2014review,chaintron2022propagation,chaintron2022propagation2}.

From the perspective of numerical simulation, the bottleneck is the $O(N^2)$ computational cost.  An efficient algorithm that reduces the computational cost to  $\mathcal{O}(N)$ is the Random Batch Method (RBM), introduced by Jin et al. \cite{jin2020random} in 2020. The RBM constructs a randomly decoupled system where interactions occur only within small subsystems of $p$ particles, with $p\ll N$. At each time step, particles are randomly divided into batches of size $p$, and interactions are computed only within these batches. Without loss of generality, we set $p$ divisible by  $N$ in this paper. We present the RBM in Algorithm \ref{algo: the RBM}, where the notation $\xi_n(i)$ is defined in Section \ref{sec:mainresult}, specifically in equation \eqref{ntt:xini}.
\begin{algorithm}\label{algo: the RBM}
\caption{ The RBM for \eqref{eq:particle}}
\For{$n=0$ to $T/\tau-1$}
    {Divide $\{1,\cdots, N\}$ into $\lceil N/p \rceil$ batches randomly\;
        Update $\tilde{X}_i^\xi$, $i=1,\cdots,N$, by solving 
        \begin{equation}\label{eq:rbmalgo}
            d\tilde{X}_i^\xi = b_0(\tilde{X}_i^\xi)dt + \frac{1}{p-1}\sum\limits_{j\in\xi_n(i),\,j\ne i}^p b(\tilde{X}_i^\xi - \tilde{X}_j^\xi)dt + \sqrt{2\sigma} dW_i,
        \end{equation}
        for $t\in [t_n,t_{n+1}).$ 
        }
\end{algorithm}

Due to the randomness in batch selection, the time-averaged effect of these interactions provides a good approximation of the original system \cite{jin2020random,ha2021uniform,jin2022random}. Thanks to its simplicity, the RBM has found applications in a variety of fields, including solving the Poisson-Nernst-Planck, Poisson-Boltzmann, and Fokker-Planck-Landau equations \cite{li2022some,carrillo2021random}, efficient sampling \cite{li2020random}, molecular dynamics simulations \cite{jin2021randomEwald,gao2024rbmd}, and quantum Monte Carlo methods \cite{LiXiantao2020random}. See also the review article \cite{jin2021randomreview}. In addition, the RBM system can be regarded as a {\it new} model with its own distinct physics. The random interaction structure within the subsystems indicates a different dynamics compared to the original system. In 2024, Li et al. demonstrated that in Landau-type models, the random batch interaction has physical significance and leads to the discretization of the diffusion term in Landau-type equations \cite{du2024collision}.

There have been several theoretical advancements in the analysis of the RBM. In the original work \cite{jin2020random}, the bounds of the strong and weak errors are presented under suitable assumptions. 
The geometric ergodicity and long-time behavior of the interacting particle systems with the RBM are investigated in \cite{jin2023ergodicity} and \cite{ye2024error}.  These works use the Wasserstein distance between the invariant distributions of the original interacting particle system and its RBM-discretized counterpart.

Different from the propagation of chaos of common particle system, there are two parameters $(N,\tau)$ in this problem. To clarify it, we show the limit relations in \eqref{graph}.
\begin{equation}\label{graph}
    \begin{tikzcd}
    \text{Original system} \,\eqref{eq:particle}  \arrow[r, "(d)\,N\to\infty"]  & \text{mean-field limit} \,\eqref{eq:mckean0}  \\
    \text{Random Batch model} \,\eqref{eq:rbmalgo} \arrow[ru,"(e)"]  \arrow[r, "(b)\,N\to\infty"']  \arrow[u, "(a)\,\tau \to 0"]& \text{mean-field limit of RBM} \arrow[u, "(c)\,\tau\to 0"']
    \end{tikzcd}
\end{equation}

We have already discussed paths (a) and (d) in the preceding paragraphs, supported by relevant references. In 2022, Jin and Li investigated the mean-field limit of the Random Batch Method (RBM) for first-order systems with Gaussian noise \cite{jin2022mean}. Under appropriate assumptions, they analyzed \eqref{eq:rbm1d} and established an $\mathcal{O}(1/N)$ estimate for path (b) and an $\mathcal{O}(\tau)$ estimate for path (c), both measured in the Wasserstein-1 distance. These results provide a rigorous foundation for understanding the convergence properties of the RBM in the context of mean-field limits. In 2024, utilizing the logarithmic Sobolev inequality, Huang et al.\cite{huang2025mean} derived a uniform-in-time $\mathcal{O}(1/N + \tau^2)$ bound on the scaled realtive entropy between the $N$-particle joint law and the tensorized law of the mean-field limit for path (e).

Specifically, in this work, we derive the bound 
\begin{equation*}
    H(\mu_t^k \mid \bar{\mu}_t^{\otimes k}) \le Ct \e^{Ct}(k^2/N^2 + k \tau^2),
\end{equation*}
for the relative entropy $H$ of the $k$-particle marginal $\mu_t^k$,    which will be defined   in Section \ref{sec:mainresult}.  When $k=1$, it leads to an $\mathcal{O}(1/N^2 + \tau^2)$ of the scaled relative entropy. Our result for $N$ is sharp and thus extends the findings of \cite{huang2025mean}. Although our result is not uniform in time, our proof does not rely on the logarithmic Sobolev inequality. 

The primary challenge arises from the coupling of the time-discretized structure in the RBM. At each time step, the particle distribution is influenced by an independent random division, which prevents the direct application of the original BBGKY hierarchies from \cite{lacker2023hierarchies} in the path space. Inspired by \cite{grass2024sharp} and \cite{du2024collision}, we use the Liouville equation to derive the hierarchies for the time-marginal distributions $\mu_t^k$. Here, a key technique is the introduction of an auxiliary random batch division, which effectively decouples particle interactions. This construction implicitly relies on the proximity between the RBM and original systems, ensuring that their dynamical behaviors remain close. In addition, we employ an alternative improvement process over \cite{grass2024sharp}, which relies on the Law of Large Numbers, to bridge the gap induced by the discrete numerical scheme. By combining these tools, we circumvent the structural complexities of the RBM and achieve our main result.

The rest of the paper is organized as follows.
In Section \ref{sec:mainresult}, we present our main theorem along with the necessary settings and notations. Section \ref{sec:xini} discusses the $k$-marginal distribution dynamics driven by the random batch system. A key coupling of the random batch division is introduced to describe the difference between the random batch system and the original  $N$-particle system. Based on the error estimates in Section \ref{sec:xini}, we give the proof of Theorem \ref{thm:1} in Section \ref{sec:proof}, while the proofs of auxiliary lemmas are provided in Section \ref{sec:auxlem}.

\section{Setup and main result}\label{sec:mainresult}


Before proceeding, we first clarify our settings and notations. 

Recall the dynamics of the RBM. For each time step $t_n$, we denote $\xi_n$ as a random division of $\{1,\cdots,N\}$ and let $\xi := (\xi_1,\cdots)$ represent the sequence of batch divisions. We take $\Omega$ as the sample space equipped with the uniform probability measure $\mathbb{P}$, and define the filtration $\{\mathcal{F}_n\}_{n\ge 0}$, where $\mathcal{F}_n$ is the $\sigma$-algebra generated by $\{\xi_j,\,j\le n\}$ and the Brownian motions before time $t_n$ in the particle system. In what follows, we will use the symbol $\E$ to indicate expectation over this probability space. For notational convenience, we slightly abuse $\xi_n(\cdot)$ to denote a set-valued function such that
\begin{equation}\label{ntt:xini}
    \xi_n(i):= \{i,i_1,\cdots,i_{p-1}\},
\end{equation}
where $\{i,i_1,\cdots,i_{p-1}\}$ belongs to a single batch in $\xi_n$. Then, the RBM dynamics for $t\in[t_n,t_{n+1})$ with $\xi_n$ is rewritten as
\begin{equation}\label{eq:rbm1d}
    d\tilde{X}_i^\xi = b_0(\tilde{X}_i^\xi)dt + \frac{1}{p-1}\sum\limits_{j\in\xi_k(i),\,j\ne i}^p b(\tilde{X}_i^\xi - \tilde{X}_j^\xi)dt + \sqrt{2\sigma} dW_i.
\end{equation}
For a fixed batch division sequence $\xi$, we denote the joint law of the system at time $t$ 
\begin{equation*}
    \mu_t^{N,\xi}:= \text{Law}(\tilde{X}^\xi (t)),
\end{equation*}
where $\tilde{X}^\xi = (\tilde{X}^\xi_1,\cdots,\tilde{X}^\xi_N).$
For the dynamics in \eqref{eq:rbm1d}, the corresponding Liouville equation governing the evolution of $\mu_t^{N,\xi}$ is given by
\begin{equation}\label{eq:rbmLiouville}
    \partial_t \mu_t^{N,\xi} = - \sum\limits_{i=1}^N \nabla_{x_i} \cdot \Big[\big(b_0(x_i) + \frac{1}{p-1}\sum_{\substack{\ell \in \xi_n(i)\\ \ell\ne i}}b(x_i - x_\ell)\big)\mu_t^{N,\xi}\Big] + \sigma\sum\limits_{i=1}^N \Delta_{x_i}\mu_t^{N,\xi}.
\end{equation}
Here, $\nabla_{x_i}\cdot$ and $\Delta_{x_i}$ denote the divergence and Laplacian operators with respect to the position $x_i$ of the $i$-th particle, respectively. 

Taking into consideration of all the possibility of the realization of random batches, the (time-marginal) law of the system described by \eqref{eq:rbm1d} is given by
\begin{equation*}
    \mu_t^N = \E_\xi \mu_t^{N,\xi},
\end{equation*}
where the expectation $\E_\xi$ is taken over all possible random batch divisions $\xi$. 

Recall the McKean SDEs,
\begin{equation}\label{eq:mckean}
    d\bar{X}_i = b_0(\bar{X}_i)dt + b*\bar{\mu}_t (\bar{X}_i)dt + \sqrt{2\sigma} dW_i, \quad i=1,\cdots,N,
\end{equation}
where $\bar{\mu}_t$ denotes the law of $\bar{X}_i$ for any $i$. Then, $\bar{\mu}_t$ satisfies the following Fokker-Planck equation: 
\begin{equation}\label{eq:FP}
    \partial\bar{\mu}_t = -\nabla \cdot [(b_0(x) + b*\bar{\mu}_t(x))\bar{\mu}_t] + \sigma\Delta\bar{\mu}_t.
\end{equation}

To quantify the discrepancy between two distributions, we utilize the relative entropy defined by
\begin{equation*} 
H\left(P \mid Q \right) :=
\left\{\begin{aligned}
    \int_E \log&\frac{dP}{dQ} dP, \quad& P\ll Q,\\
    &\infty, \quad& \text{otherwise,}
\end{aligned}\right.
\end{equation*}
where $P$ and $Q$ are two probability measures over some appropriate space $E$.

For the $k$-particle marginal distribution, we define
\begin{equation*}
    \mu_t^k := \int_{\R^{(N-k)d}} \mu_t^N dx_{k+1}\cdots dx_N,
\end{equation*}
and denote the relative entropy of $\mu_t^k$ with respect to the tensorized mean-field distribution 
\begin{equation}
    H_t^k:=H(\mu_t^k\,|\, \bar{\mu}_t^{\otimes k}).
\end{equation}
In this work, we assume the integral domain to be the entire space by default and omit the corresponding subscript for simplicity.

We first present our assumptions of the coefficients and the initial distribution.
\begin{assumption}\label{ass:b0b1}
The drift terms $b_0$ and $b_1$ satisfy
\begin{enumerate}
        \item The external drift $b_0$ is one-sided Lipschitz in the sense that there exists $L_0 \in \R$ such that
    \begin{gather}\label{eq:onesidedlip}
    (x-y)\cdot (b_0(x)-b_0(y))\le L_0 |x-y|^2.
    \end{gather}
    \item The interaction kernel $b$ is a Lipschitz smooth function with bounded second order derivatives.
\end{enumerate}
\end{assumption}

\begin{remark}
Note that the one-sided Lipschitz condition \eqref{eq:onesidedlip} on $b_0$ guarantees the weak  well-posedness of the SDEs \eqref{eq:mckean} and \eqref{eq:rbm1d} and allows some confining fields that may have superlinear growth such as $b_0(x)=-|x|^qx$ with $q\ge 1$. In fact, one may allow more relaxed conditions to require that there exists a unique in law weak solution of the SDE
 \begin{equation*}
        dY(t) = b_0(Y(t))dt + \sqrt{2\sigma}dW(t), \quad Y(t) =y_t,
\end{equation*}
 where $W$ is a $d$-dimensional Brownian motion. This assumption, however, does not provide explicit verifiable conditions.
\end{remark}

\begin{assumption}\label{ass:ini}
The initial distribution $\mu_0$ satisfies
    \begin{enumerate} 
        \item The initial Fisher information is finite, i.e., $\int_{\mathbb{R}^d} |\nabla \log (\mu_0)|^2\mu_0 <\infty$.
    \item The initial distribution $\mu_0$ is sub-Gaussian; that is, for a random variable $Y(0)\sim \mu_0$, there exists a constant $C$ such that for any $t>0$,
    \begin{equation*}
        \mathbb{P} (|Y(0)|\ge t) \le 2 \exp (-t^2/C^2).
    \end{equation*}
    \end{enumerate}
\end{assumption}


Then, we turn to our main theorem.

\begin{theorem}\label{thm:1}
 Suppose that $b_0$ and $b_1$ satisfy Assumption \ref{ass:b0b1}, the initial law $\mu_0$ satisfies Assumption \ref{ass:ini}, and $\sigma>0$. Consider the joint law $\mu_t^N$ for the $N$-particle system \eqref{eq:rbm1d} and the law $\bar{\mu}_t$ for system \eqref{eq:mckean}. Then, for any $t>0$, there exists a constant $C$ independent of $N$ and $k$ such that
\begin{equation}\label{eq:thm1}
    H(\mu_t^k\,|\, \bar{\mu}_t^{\otimes k}) \le Ct\e^{Ct}\left(\frac{k^2}{N^2} + k\tau^2\right).
\end{equation}
\end{theorem}


Based on Theorem \ref{thm:1} and Pinsker's inequality \cite{pinsker1964information}, we are able to extend the propagation of chaos to that under total variation (TV) distance defined by
\begin{equation*}
     \|\mu -\nu\|_{TV} := \sup_{A\in \mathcal{E}} |\mu(A) - \nu(A)|,
\end{equation*}
for two probability measures $\mu$, $\nu$ defined on certain measure space $(E,\mathcal{E})$.
\begin{corollary}\label{coro:TV}
    Under the same settings of Theorem \ref{thm:1},  for $1 \leq k \leq N$, it holds that
    \begin{equation*}
            \|\mu_t^k - \bar{\mu}^{\otimes k}_t \|_{TV}\le C\sqrt{t}e^{Ct} \left(\frac{k}{N}+ \sqrt{k}\tau\right).
    \end{equation*}
\end{corollary}

\begin{remark}
    The Assumption \ref{ass:b0b1} implies that the kernel $b$ is allowed to have a linear growth. In this sense, our main result accommodates some unbounded kernel $b$ under appropriate regularity conditions, thereby relaxing the conventional boundedness assumption of $b$ in pervious RBM analysis \cite{jin2022mean,du2024collision,huang2025mean}.
\end{remark}

\begin{remark}
     If the interaction kernel $b$ is bounded (as in \cite{huang2025mean,grass2024sharp}), the sub-Gaussian assumption on the initial distribution $\mu_0$ can be relaxed. In this case, the second line of \eqref{eq:estimateIk} directly yields the estimate
    \begin{equation*}
        \int \mathbb{I}_k \log \Big(\frac{\mu_t^k}{\bar{\mu}_t^{\otimes k}}\Big)
   \le  \epsilon_3 I_t^k + \frac{C}{\epsilon_3}\e^{Ct}k\tau^2,
    \end{equation*}
    by estimating the term 
    \begin{equation*}
        \E \int |\mu_{t,n}^{N,\xi_n} - \mu_{t,n}^{N,\xi_n^i}|^2\,/\,(\mu_{t,n}^{N,\xi_n} + \mu_{t,n}^{N,\xi_n^i}),
    \end{equation*}
    in a similar way with Proposition \ref{prop:5.1}, without requiring the moment control provided by Lemma \ref{lem:momentcontrol}.
\end{remark}

\section{The dynamics of the RBM-driven distribution}\label{sec:xini}

As we mentioned in the introduction, the discrete randomness induced by the RBM prevents the direct application of the BBGKY hierarchies for path measures developed in \cite{lacker2023hierarchies,lacker2023sharp}. Consequently, we instead consider the hierarchies of the time-marginal distributions.

First, we derive the Liouville-type equations governing the dynamics of $\mu_t^k$.
Then, we establish a crucial local error estimate by introducing a coupling of the random batch division, serving as the preparation for the proof of Theorem \ref{thm:1}.

\subsection{Derivation of the dynamics of $k$-particle marginal distribution}

First, to deal with the time-discretized random structure of \eqref{eq:rbmLiouville}, we introduce an auxiliary distribution $\mu_{t,n}^{N,\xi_n}$ defined on time interval $[t_n,t_{n+1})$ and derive the PDE governing $\mu_t^k$ with respect to $\mu_{t,n}^{N,\xi_n}$.

Now we consider the distribution $\mu_{t,n}^{N,\xi_n}$ for $t\in[t_n, t_{n+1}]$ given by
\begin{equation*}
\mu_{t,n}^{N,\xi_n}=\E[\mu_{t}^{N,\xi} | \xi_m, m\ge n].
\end{equation*}
This means that the batch at $t_n$ is fixed to be $\xi_n$ while the randomness brought by batches before $t_n$ are averaged out. By the linearity of the Fokker-Planck equation, it is straightforward to see
\begin{equation*}
    \mu_{t_n,n}^{N,\xi_n} = \mu_{t_n}^N,
\end{equation*}
and for $t\in [t_n,t_{n+1}),$ $\mu_{t,n}^{N,\xi_n}$ satisfies
\begin{equation}\label{eq:rbmLmutn}
    \partial_t \mu_{t,n}^{N,\xi_n} = - \sum\limits_{i=1}^N \nabla_{x_i}\cdot\Big[\big(b_0(x_i) + \frac{1}{p-1}\sum_{\substack{\ell \in \xi_n(i)\\ \ell\ne i}}b(x_i - x_\ell)\big)\mu_{t,n}^{N,\xi_n}\Big] + \sigma\sum\limits_{i=1}^N \Delta_{x_i}\mu_{t,n}^{N,\xi_n}.
\end{equation}

Note that if we observe the random batch particles system at the discrete time $t_n$, it then forms a time-homogeneous Markov chain. It is clear that the law of $\tilde{X}$ at $t_n$ is given by 
\begin{equation*}
    \mu_{t_n}^N = \E_\xi \mu_{t_n}^{N,\xi}.
\end{equation*}
Hence, $\mu_{t,n}^{N,\xi_n}$ is just the law of the random batch particles with the batch at $t_n$ prescribed. By the Markov property, we can understand that the positions at $t_n$ are drawn from $\mu_{t_n}^N$, and they evolve according to the batch $\xi_n$ during $[t_n, t_{n+1}]$.

With the above understanding, it is clear that 
\begin{equation*}
\mu_t^N = \E_{\xi_n}\mu_{t,n}^{N,\xi_n} = \E_\xi \mu_t^{N,\xi},
\end{equation*}
for $t\in[t_n,t_{n+1})$.
Then one has
\begin{multline}\label{eq:mutk1}
\partial_t \mu_t^N 
= -\sum\limits_{i=1}^N \nabla_{x_i}\cdot \big(b_0(x_i)\mu_t^N\big)\\
-\sum\limits_{i=1}^N\frac{1}{p-1}\nabla_{x_i}\cdot\E_{\xi_n} \Big[\sum_{\substack{\ell \in \xi_n(i)
\ell\ne i}}b(x_i - x_\ell)\mu_{t,n}^{N,\xi_n}\Big] + \sigma\sum\limits_{i=1}^N\Delta_{x_i}\mu_t^N.
\end{multline}
Recalling $\mu_t^k:=\int \mu_t^N dx_{k+1}\cdots dx_N$, it holds
\begin{multline}\label{eq:mutkstarting}
    \partial_t \mu_t^k = - \sum\limits_{i=1}^k \nabla_{x_i}\cdot (b_0(x_i)\mu_t^k)+\sigma\sum\limits_{i=1}^k\Delta_{x_i}\mu_t^k\\
    - \sum\limits_{i=1}^k \int \frac{1}{p-1}  \nabla_{x_i}\cdot \E_{\xi_n}\Big[(\sum_{\substack{\ell \in \xi_n(i)\\ \ell\ne i}}b(x_i - x_\ell))\mu_{t,n}^{N,\xi_n}\Big]dx_{k+1}\cdots dx_N.
\end{multline}

\subsection{Introduction of auxiliary random batch divisions}\label{subsec:coupled sys}

Secondly, we decouple the system by introducing auxiliary random divisions. This is a key construction of our proof. Given a random batch division $\xi_n$ and a fixed $i$, we define the following auxiliary division $\xi_n^i$:
\begin{itemize}
    \item Draw a random variable $\zeta\sim \mathrm{UNIF}(0, 1)$, the uniform distribution on $(0, 1)$.
    \item If $\zeta\le (p-1)/(N-1)$, we set $\xi_n^i = \xi_n$.
    \item If $\zeta>(p-1)/(N-1)$, choose $j$ from $\{1,\cdots,N\}\backslash \xi_n(i)$ randomly with equal probability. Exchange the $p-1$ batchmates of Particle $j$ with the batchmates of Particle $i$, resulting the new division of batches $\xi_n^i$. In other words, based on
\begin{equation*}
    \xi_n(i) = \{i,i_1,\cdots,i_{p-1}\}, \quad \xi_n(j) = \{j,j_1,\cdots,j_{p-1}\},
\end{equation*}
we set
\begin{equation*}
    \xi_n^i(i) := \{i,j_1,\cdots,j_{p-1}\},\quad \xi_n^i(j) := \{j,i_1,\cdots, i_{p-1}\}.
\end{equation*}
\end{itemize}
The above just says we keep the batch division unchanged with probability $(p-1)/(N-1)$, and otherwise switch the batchmates of a random particle $j\notin \xi_n(i)$ with those of $i$. This auxiliary division satisfies the following useful property, which in fact reflects the intrinsic property of the random batches.

\begin{proposition}\label{lem:xini}
The coupled batch division $\xi_n^i$ is a random batch division in the sense that it has the same law of $\xi_n$.
In addition, for any $j\ne i$, and $\xi_n^i$ defined by Section \ref{subsec:coupled sys}, one has
    \begin{equation*}
    \mathbb{P}(j\in\xi_n(i)\mid \xi_n^i) = \frac{p-1}{N-1}.
\end{equation*}
\end{proposition}
 \begin{proof}
  The first claim follows from the symmetry of the construction because there is no preference of batchmates among the different particles.

  Now, consider the second claim. Fix a possible value $\mathfrak{s}=\{\mathfrak{s}(1),\cdots, \mathfrak{s}(N)\}$ of $\xi_n^i$. Here, $\mathfrak{s}(i)$ indicates the group of size $p$ for particle $i$.
        
   If $j\in \mathfrak{s}(i)$, then by Bayesian's formula,
    \begin{equation}\label{eq:lemxini1}
        \mathbb{P}(j\in\xi_n(i)\mid \xi_n^i=\mathfrak{s}) = \frac{\mathbb{P}(j\in \xi_n(i), \xi_n^i=\mathfrak{s})}{\mathbb{P}(\xi_n^i=\mathfrak{s})}=\frac{\mathbb{P}(\xi_n=\mathfrak{s}, \xi_n^i=\xi_n)}{\mathbb{P}(\xi_n^i=\mathfrak{s})} = \frac{p-1}{N-1}.
    \end{equation}
    If $j\notin\mathfrak{s}(i)$, one has, by Bayesian's formula, that
    \begin{equation}\label{eq:lemxini2}
        \begin{aligned}
        \mathbb{P}(j\in\xi_n(i)\mid \xi_n^i=\mathfrak{s}) =& \frac{\mathbb{P}(j\in \xi_n(i), \xi_n^i=\mathfrak{s})}{\mathbb{P}(\xi_n^i=\mathfrak{s})}\\
        =&\sum_{\mathfrak{s}': j\in \mathfrak{s}'(i)}\frac{\mathbb{P}(\xi_n=\mathfrak{s}')\mathbb{P}(\xi_n^i=\mathfrak{s}|\xi_n=\mathfrak{s}')}{\mathbb{P}(\xi_n^i=\mathfrak{s})}\\
        =& \sum_{\mathfrak{s}': j\in \mathfrak{s}'(i)}\mathbb{P}(\xi_n^i=\mathfrak{s}|\xi_n=\mathfrak{s}').
    \end{aligned}
    \end{equation}
To compute this probability, we note that $\mathbb{P}(\xi_n^i=\mathfrak{s}|\xi_n=\mathfrak{s}')$ is nonzero for only $p-1$ possible $\mathfrak{s}'$ (one takes $k\in\mathfrak{s}(j)\setminus\{j\}$, then $\mathfrak{s}'(i)=\{i\}\cup \mathfrak{s}(j)\setminus\{k\}$ and $\mathfrak{s}'(k)=\{k\}\cup \mathfrak{s}(i)\setminus\{i\}$). For each such $\mathfrak{s}'$, 
\[
\mathbb{P}(\xi_n^i=\mathfrak{s}|\xi_n=\mathfrak{s}')=\left(1-\frac{p-1}{N-1}\right)*\frac{1}{N-p}=\frac{1}{N-1}.
\]
This then proves the result for the second case.
Now the proof is  finished.
\end{proof}

\begin{remark}
For the second case, an intuitive way is to consider
\[
\mathbb{P}(j\in\xi_n(i)\mid \xi_n^i=\mathfrak{s})
=\mathbb{P}(\xi_n\neq \xi_n^i \mid \xi_n^i=\mathfrak{s})
\mathbb{P}(j\in\xi_n(i)\mid \xi_n^i=\mathfrak{s}, \xi_n\neq \xi_n^i)
\]
Intuitively, the probability $\mathbb{P}(\xi_n\neq \xi_n^i \mid \xi_n^i=\mathfrak{s})=(N-p)/(N-1)$ because for every occurance of random batch, the probability $\xi_n\neq \xi_n^i$ is the same. The second probability $\mathbb{P}(j\in\xi_n(i)\mid \xi_n^i=\mathfrak{s}, \xi_n\neq \xi_n^i)$ should be $(p-1)/(N-p)$ due to a symmetry consideration as one should choose $p-1$ particles from $\{1,\cdots, N\}\setminus \mathfrak{s}(i)$ to make the batchmates of $i$. The probability that particle $j$ is selected should be $(p-1)/(N-p)$. The rigorous justification of this, however, should then be referred  to the Bayesian formula as above.
\end{remark}

The above fact about the constructed $\xi_n^i$ is crucial to deal with the interacting term. By Proposition \ref{lem:xini}, for each $j\ne i$,  one has
\begin{equation*}
    \mathbb{P}(j\in\xi_n(i)\mid \xi_n^i) = \frac{p-1}{N-1}.
\end{equation*}
Then,
\begin{multline}
\E_{\xi_n}\Big[\Big(\sum_{\substack{\ell \in \xi_n(i)\\ \ell\ne i}}b(x_i - x_\ell)\Big)\mu_{t,n}^{N,\xi_n}\Big]=\\
\E\Big[(\sum_{\substack{\ell \in \xi_n(i)\\ \ell\ne i}}b(x_i - x_\ell))\mu_{t,n}^{N,\xi_n^i}\Big]
+\E\Big[(\sum_{\substack{\ell \in \xi_n(i)\\ \ell\ne i}}b(x_i - x_\ell))(\mu_{t,n}^{N,\xi_n}-\mu_{t,n}^{N,\xi_n^i})\Big],
\end{multline}
where the expectation on the right hand side is the one on the probability space for $\xi_n$ and $\xi_n^i$, namely $\E_{\xi_n,\xi_n^i}$.
The first term can then be computed by
\begin{multline*}
 \E \Big[\frac{1}{p-1}\sum_{\substack{\ell \in \xi_n(i)\\ \ell\ne i}}b(x_i - x_\ell)\mu_t^{N,\xi_n^i}\Big]
 = \E \Bigg[\E\Big[ \big(\frac{1}{p-1}\sum_{\substack{\ell \in \xi_n(i)\\ \ell\ne i}}b(x_i - x_\ell)\mu_t^{N,\xi_n^i}\big)|\xi_n^i \Big]\Bigg]\\
 = \E \Big[\nabla_{x_i}\cdot \big(\frac{1}{N-1}\sum\limits_{j=1,\,j\ne i}^N b(x_i - x_j)\mu_t^{N,\xi_n^i}\big)\Big]
        =\frac{1}{N-1}\sum\limits_{j=1,\,j\ne i}^N  \Big(b(x_i - x_j)\mu_t^N\Big).
\end{multline*}

Define
\begin{equation*}
    \mathbb{I}_k : = \sum\limits_{i=1}^N \int \frac{1}{p-1}\E \Big[\nabla_{x_i}\cdot \Big(\sum_{\substack{\ell \in \xi_n(i)\\ \ell\ne i}}b(x_i - x_\ell)(\mu_{t,n}^{N,\xi_n}-\mu_{t,n}^{N,\xi_n^i})\Big)\Big]dx_{k+1}\cdots dx_N.
\end{equation*}
Noting that $\mu_t^N$ is a symmetric measure, one  can then rewrite \eqref{eq:mutkstarting} as
\begin{equation}\label{eq:mutk3}
\begin{split}
\partial_t \mu_t^k =& -\sum\limits_{i=1}^k \nabla_{x_i}\cdot \Big[\big(b_0(x_i)+\frac{1}{N-1}\sum\limits_{j=1,\,j\ne i}^kb(x_i - x_j)\big)\mu_t^k\Big]\\
        &- \sum\limits_{i=1}^k \nabla_{x_i}\cdot \Big[\frac{N-k}{N-1}\int b(x_i - x_{k+1})\mu_t^{k+1}dx_{k+1}\Big]\\
        &+ \sum\limits_{i=1}^k \sigma\Delta_{x_i}\mu_t^{k} - \mathbb{I}_k\\
        =&: \mathcal{L}_{N,k} \mu_t^N - \mathbb{I}_k.
\end{split}
\end{equation}
Here, we define the Liouville-type operator $\mathcal{L}_{N,k}$ for notational convenience. From \eqref{eq:mutk3}, it is evident that the law $\mu_t^N$ exhibits a structure similar to the case without the RBM, except for the additional residual terms $\mathbb{I}_k$.

\subsection{The discrepancy between the coupled random divisions}
Under the construction of $\xi_n^i$, we estimate the error terms arising from $\mu_{t,n}^{N,\xi_n^i}$. The proof relies on two lemmas, the details of which are provided in Section \ref{sec:auxlem}.
\begin{proposition}\label{prop:5.1}
    Under the assumptions of Theorem \ref{thm:1}, for any $t_n \le t \le \min \left\{t_{n+1}, T\right\}$ and any given $i$, the following holds
\begin{equation}\label{eq:prop5.1}
    \mathbb{E}  \int \left|\mu_{t,n}^{N,\xi_n}-\mu_{t,n}^{N, \xi_n^i}\right|^4 \,/\, \left(\mu_{t,n}^{N, \xi_n}+\mu_{t,n}^{N, \xi_n^i}\right)^3 \le C \e^{Ct} \tau^4,
\end{equation}
where the constant $C$ is independent of $N$.\\
\end{proposition}

\begin{proof}
Denote
\begin{equation*}
    E(t):=\E \sum\limits_{i=1}^N \int \frac{\left|\mu_{t,n}^{N,\xi_n}-\mu_{t,n}^{N, \xi_n^i}\right|^4}{\left(\mu_{t,n}^{N, \xi_n}+\mu_{t,n}^{N, \xi_n^i}\right)^3}.
\end{equation*}
At the time point $t_n$, since $\mu_{t,n}^{N,\xi}$ (and also $\mu_{t,n}^{N,\xi_n^i}$) is equal to $\mu_{t_n}^N$,  $E(t_n) =0$. Below, we will basically show that
\begin{equation}\label{eq:prop1:ode}
    \frac{d}{dt}E(t)^{\frac{1}{4}} \le C N^{\frac{1}{4}}\e^{Ct}.
\end{equation}
Then the symmetry with respect to $i$ gives \eqref{eq:prop5.1}.

Now we derive \eqref{eq:prop1:ode}. For notation simplicity, we fix $i$ and denote
\begin{align*}
& f:=\mu_{t}^{N, \xi_n}, \quad \tilde{f}:=\mu_{t}^{N,\xi_n^i},  \quad \delta f:=f-\tilde{f}; \\
& b^{\xi_n}:=
    \begin{bmatrix}
        b_0(x_1) \\
        \vdots \\
        b_0(x_N)
    \end{bmatrix}
    +\frac{1}{p-1}
    \begin{bmatrix}
        \sum\limits_{\ell \in \xi_n(1)} b\left(x_1 - x_\ell\right) \\
        \vdots \\
        \sum\limits_{\ell \in \xi_n(N)} b\left(x_N - x_\ell\right)
    \end{bmatrix},\quad
 \delta b^i:=b^{\xi_n}-b^{\xi_n^{i}}.
\end{align*}
By straightforward computation, one has that
\begin{equation}\label{eq:ddt prop5.1}
\frac{d}{d t} \int \frac{|\delta f|^4}{(f+\tilde{f})^3} 
 \le -8 \int \left(\frac{\delta f}{f + \tilde{f}}\right)^3 \nabla_x \cdot \left(\delta b^i \frac{f\tilde{f}}{f+\tilde{f}}\right).
\end{equation}
We leave the computational details in Appendix \ref{app:prop5.1}.
By H\"older's inequality, it holds
\begin{equation}\label{eq:est prop5.1}
\frac{d}{d t} E(t) \le 8(E(t))^{\frac{3}{4}}\left\{\mathbb{E} \sum_{i=1}^{N} \int \frac{1}{(f+\tilde{f})^3}\left(\sum_{j=1}^{N} \nabla_{x_j} \cdot\left(\frac{f \tilde{f}}{f+\tilde{f}} \delta b^i\right)\right)^4\right\}^{\frac{1}{4}}.
\end{equation}

Recall the notation \eqref{ntt:xini} and set the randomly selected particle from $\{1,\cdots,N\}\backslash\{i\}$ to be $\tilde{i}$. Note that in $\delta b^i$, only the rows of $\xi_n(i)\cup\xi_n(\tilde{i})$ are potentially nonzero terms. We denote $\xi_n(\tilde{i}): = \{\tilde{i},\tilde{i}_1,\cdots,\tilde{i}_{p-1}\}$ and $\Lambda_i : = \xi_n(i)\cup \xi_n(\tilde{i})$.
Then, it holds
\begin{align}
\notag& \left| \sum_{j=1}^{N} \nabla_{x_{j}} \cdot\left(\frac{f \tilde{f}}{f+\tilde{f}} \delta b^i\right)\right|^4 \frac{1}{(f+\tilde{f})^3} \\
\notag& \lesssim \left| \sum_{j \in \Lambda_i}\left[|f \tilde{f}|\frac{\|\nabla b\|_\infty}{p-1} + |f \tilde{f}_j\delta b^i |+|f_j \tilde{f}\delta b^i | + |\delta b^i| |f_j+\tilde{f}_j| f \right]\right|^4 \left(\frac{1}{f+\tilde{f}}\right)^7 \\
& \lesssim \frac{1}{p} \sum_{j \in \Lambda_i}\left(\left| f \right|+\frac{|f_j (p-1) \delta b^i |^4}{(f+\tilde{f})^3} + \frac{|\tilde{f}_j (p-1) \delta b^i |^4}{(f+\tilde{f})^3}\right),\label{aux eq:prop5.1}
\end{align}
where the subscript $j$ of $f$ denotes $\frac{\partial}{\partial x_j}$ of $f$. Here, $\lesssim \cdot$ means $\le C \cdot$ for some constant $C$.
Note that by H\"older's inequality,
\begin{multline}
    \E \sum\limits_{i=1}^N \frac{1}{p}\int\sum_{j \in \Lambda_i}\frac{|f_j (p-1) \delta b^i |^4}{(f+\tilde{f})^3} \\
    \le \sum\limits_{i=1}^N \frac{1}{p} \left( \E \sum\limits_{j\in \Lambda_i} \int \frac{|f_j|^8}{(f+\tilde{f})^7}\right)^{\frac{1}{2}}\left( \E \sum\limits_{j\in \Lambda_i} \int |(p-1)\delta b^i |^8 (f+\tilde{f})\right)^{\frac{1}{2}}.
\end{multline}
We consider the first term
\begin{equation*}
    \int \frac{|f_j|^8}{(f+\tilde{f})^7} \le \int \frac{|f_j|^8}{f} = \int |\nabla_{x_j}\log f|^8 f.
\end{equation*}
By Lemma \ref{lem:5.1}, for $t\in[t_n,t_{n+1})$, one has 
\begin{equation*}
\begin{aligned}
    \E \sum\limits_{\ell \in \xi_n(j)} \int \frac{|f_\ell |^8}{f} \le& Cp(t-t_n) + \E\,\e^{C(t-t_n)}\sum\limits_{\ell\in\xi_n(j)} \hat{I}_8^\ell(t_n)\\
    \le& Cp(t-t_n) + C\E\frac{p}{N}\sum\limits_{\ell=1}^N\int \frac{|f_\ell|^8}{f}(t_n),
\end{aligned}
\end{equation*}
since $\xi_n$ is independent of $f(t_n)$.
Then, by Lemma \ref{lem:Cor5.1}, it holds
\begin{equation*}
    \E\frac{p}{N}\sum\limits_{\ell=1}^N\int \frac{|f_\ell|^8}{f^7}(t_n) \le C p \e^{Ct_n},
\end{equation*}
and then
\begin{equation*}
    \E \sum\limits_{i=1}^N \frac{1}{p}\sum\limits_{j\in \Lambda_i}\int \frac{|f_j|^8}{f}(t)\le CN(t-t_n) + CN\e^{Ct_n}\le C\e^{Ct}N,
\end{equation*}
since there are at most $2p$ particles contained in $\Lambda_i$. 

In addition, combining the moments control in Lemma \ref{lem:momentcontrol} and the Lipschitz assumption of $b$, one has
\begin{equation*}
    \E \sum\limits_{j\in \Lambda_i} \int |(p-1)\delta b^i |^8 (f+\tilde{f}) \lesssim p.
\end{equation*}

Therefore, one obtains
\begin{equation}\label{eq:fprop5.1}
    \E \sum\limits_{i=1}^N \frac{1}{p}\int\sum_{j \in \Lambda_i}\frac{|f_j (p-1) \delta b^i |^4}{(f+\tilde{f})^3} \le \sum\limits_{i=1}^N \frac{1}{p} C\sqrt{p} \e^{Ct}\sqrt{p} \le CN \e^{Ct}.
\end{equation}
Since both $\xi_n$ and $\xi_n^i$ are both valid random batch divisions of the RBM, the estimate of 
\begin{equation}\label{eq:tildefprop5.1}
    \E \sum\limits_{i=1}^N \frac{1}{p}\int\sum_{j \in \Lambda_i}\frac{|f_j (p-1) \delta b^i |^4}{(f+\tilde{f})^3} \le CN \e^{Ct},
\end{equation}
is similar,

Combining the above estimates \eqref{eq:est prop5.1}, \eqref{aux eq:prop5.1}, \eqref{eq:fprop5.1} and \eqref{eq:tildefprop5.1}, one gets 
\begin{equation*}
    \frac{d}{d t} E(t) \le C E(t)^{\frac{3}{4}} N^{\frac{1}{4}} e^{Ct} \quad \Longrightarrow \quad E(t) \le C N\e^{Ct}(t-t_n)^4.
\end{equation*}
Now the proof is finished.  
\end{proof}

\section{Proof of Theorem \ref{thm:1}}\label{sec:proof}

In this section, we present the proof of Theorem \ref{thm:1} which contains three steps. First, we obtain a series of primer inequality estimate using the BBGKY hierarchies. This step needs to separately analyze the decoupled part and bound the remaining interactions based on the above section. Second, we establish  the boundedness of $H_t^N$ by the Large Derivation Principle. Finally, we obtain the convergence order using the bootstrapping technique.

\subsection{Derivation of the ODEs for the local relative entropy}\label{subsec:derivation}

since 
\begin{equation*}
    \partial_t \mu_t^k + \mathbb{I}_k = \mathcal{L}_{N,k}\mu_t^N.
\end{equation*},
 we first compute
\begin{equation*}
    \begin{aligned}
        \frac{d}{dt}H(\mu_t^k\,|\, \bar{\mu}_t^{\otimes k}) 
        =& \int \partial_t \mu_t^k \log \Big(\frac{\mu_t^k}{\bar{\mu}_t^{\otimes k}}\Big) - \partial_t \log (\bar{\mu}_t^{\otimes k})\mu_t^k\\
        =& \int \mathcal{L}_{N,k}\mu_t^N \log \Big(\frac{\mu_t^k}{\bar{\mu}_t^{\otimes k}}\Big) - \int \partial_t \log (\bar{\mu}_t^{\otimes k})\mu_t^k - \int \mathbb{I}_k \log \Big(\frac{\mu_t^k}{\bar{\mu}_t^{\otimes k}}\Big).
    \end{aligned}
\end{equation*}
Recall the operator $\mathcal{L}_{N,k}$ is defined in \eqref{eq:mutk3}.
Note that the first two terms are formally analogous to the case without the RBM. We set
\begin{equation*}
    \int \mathcal{L}_{N,k}\mu_t^N\, \log \Big(\frac{\mu_t^k}{\bar{\mu}_t^{\otimes k}}\Big) - \int \partial_t \log (\bar{\mu}_t^{\otimes k})\mu_t^k := \mathcal{A} - \mathcal{B}.
\end{equation*}
Then, by \eqref{eq:mutk3} and integration by parts, it holds that
\begin{align*}
        \mathcal{A} =& \int \mathcal{L}_{N,k}\mu_t^N \log \Big(\frac{\mu_t^k}{\bar{\mu}_t^{\otimes k}}\Big) \\
        =& \int \sum\limits_{i=1}^k \Big[\big(b_0 (x_i) + \frac{1}{N-1}\sum\limits_{j=1,\,\,j\ne i}^k b(x_i - x_j)\big) \mu_t^k\Big]\cdot \nabla_{x_i} \log \Big(\frac{\mu_t^k}{\bar{\mu}_t^{\otimes k}}\Big) \\
        &+ \int \sum\limits_{i=1}^k \frac{N-k}{N-1} b(x_i - x_{k+1})\mu_t^{k+1} \cdot \nabla_{x_i} \log \Big(\frac{\mu_t^k}{\bar{\mu}_t^{\otimes k}}\Big) \\
        &- \int \sum\limits_{i=1}^k \sigma \mu_t^k \nabla_{x_i} \log \mu_t^k \cdot \nabla_{x_i}\log \Big(\frac{\mu_t^k}{\bar{\mu}_t^{\otimes k}}\Big), 
    \end{align*}
and
    \begin{align*}
        \mathcal{B}=& \int \partial_t \log(\bar{\mu}_t^{\otimes k})\mu_t^k = \int \partial_t \bar{\mu}_t^{\otimes k} \frac{\mu_t^k}{\bar{\mu}_t^{\otimes k}}\\
        =& -\sum\limits_{i=1}^k \int \sigma \mu_t^k \nabla_{x_i}\log \bar{\mu}^{\otimes k}_t\cdot \nabla_{x_i}\log \Big(\frac{\mu_t^k}{\bar{\mu}_t^{\otimes k}}\Big)\\
        &+ \sum\limits_{i=1}^k \int \big(b_0 (x_i) + b* \bar{\mu}_t(x_i)\big)\mu_t^k\cdot \nabla_{x_i} \log \Big(\frac{\mu_t^k}{\bar{\mu}_t^{\otimes k}}\Big).
    \end{align*}
Then, the following holds
\begin{equation}\label{eq:A-B}
        \mathcal{A}-\mathcal{B} 
        = -\sigma\sum\limits_{i=1}^k \int \mu_t^k \Big|\nabla_{x_i} \log \Big(\frac{\mu_t^k}{\bar{\mu}_t^{\otimes k}}\Big)\Big|^2 + \mathcal{J}_k +\mathcal{K}_k,
\end{equation}
where
\begin{equation}\label{eq:nttJ}
    \mathcal{J}_k:= \sum\limits_{i=1}^k \int\Big( \mu_t^k \sum\limits_{j=1,\,\,j\ne i}^k
    \frac{1}{N-1}b(x_i - x_j)-\mu_t^k\frac{k-1}{N-1}b*\bar{\mu}_t(x_i)\Big)\cdot\nabla_{x_i}\log(\frac{\mu_t^k}{\bar{\mu}_t^{\otimes k}}),
\end{equation}
and
\begin{equation}\label{eq:nttK}
    \mathcal{K}_k:= \frac{N-k}{N-1}\sum\limits_{i=1}^k \int \big(\mu_t^{k+1}b(x_i - x_{k+1})- \mu_t^k b*\bar{\mu}_t(x_i)\big)\cdot \nabla_{x_i}\log(\frac{\mu_t^k}{\bar{\mu}_t^{\otimes k}}).
\end{equation}
For the term $\mathcal{J}_k$, Young's inequality yields
\begin{equation}\label{eq:Jestimate1}
    \mathcal{J}_k\le \sum\limits_{i=1}^k \epsilon_1\int \mu_t^k \bigg|\nabla_{x_i}\log(\frac{\mu_t^k}{\bar{\mu}_t^{\otimes k}})\bigg|^2 + \frac{C}{\epsilon_1 N^2}\sum\limits_{i=1}^k \int \mu_t^k \bigg|\sum\limits_{j=1,\,j\ne i}^k b(x_i - x_j) - (k-1)b*\bar{\mu}_t (x_i)\bigg|^2.
\end{equation}
For the term $\mathcal{K}_k$, we introduce the notation $\mu_t^{k+1|k}(x_{k+1}\mid x_1,\cdots,x_k)$ which represents the conditional distribution at time $t$ of the first $k+1$ particles, given the first $k$ particles.
Then, since $\|\nabla b\|_\infty <\infty$, Lemma \ref{lem:CKP} and Lemma \ref{lem:sub-Gaussian} yield
\begin{equation*}
    |\langle b(x_i - \cdot), \mu_t^{k+1|k}-\bar{\mu}_t \rangle| \le C\sqrt{H(\mu_t^{k+1|k}\mid \bar{\mu}_t)}.
\end{equation*}
Combining with Young's inequality, it holds
\begin{equation}\label{eq:Kestimate1}
    \begin{aligned}
        \mathcal{K}_k=& \frac{N-k}{N-1}\sum\limits_{i=1}^k \int \mu_t^k \nabla_{x_i} \log \Big(\frac{\mu_t^k}{\bar{\mu}_t^{\otimes k}}\Big) \cdot \left<b(x_i - \cdot),\mu_t^{k+1|k} - \bar{\mu}_t \right>\\
        \le & \frac{N-k}{N-1}\sum\limits_{i=1}^k \epsilon_2 \int \mu_t^k \Big|\nabla_{x_i} \log \Big(\frac{\mu_t^k}{\bar{\mu}_t^{\otimes k}}\Big)\Big|^2 
        +\frac{N-k}{N-1}\frac{C}{4\epsilon_2}\sum\limits_{i=1}^k \int \mu_t^k  H(\mu_t^{k+1|k}|\bar{\mu}_t).
    \end{aligned}
\end{equation}
We denote $I_t^k:= \sum\limits_{i=1}^k \int \mu_t^k |\nabla_{x_i}\log \Big(\frac{\mu_t^k}{\bar{\mu}_t^{\otimes k}}\Big)|^2$ as the relative Fisher information. 
By the chain rule of the relative entropy, one has
\begin{equation*}
    \int \mu_t^k H(\mu_t^{k+1|k}\mid \bar{\mu}_t) = H_t^{k+1} - H_t^k.
\end{equation*}
Then combining \eqref{eq:A-B}, \eqref{eq:Jestimate1} and \eqref{eq:Kestimate1}, one gets
\begin{align*}
    \frac{d}{dt}H(\mu_t^k\,|\, \bar{\mu}_t^{\otimes k}) \le& (-\sigma +\epsilon_1 + \epsilon_2)I_t^k + \frac{C}{\epsilon_1 N^2}\sum\limits_{i=1}^k \int \mu_t^k \bigg|\sum\limits_{j=1,\,j\ne i}^k b(x_i - x_j) - (k-1)b*\bar{\mu}_t (x_i)\bigg|^2\\ &+ \frac{C}{\epsilon_1}\frac{(N-k)k}{N}(H_t^{k+1}-H_t^k) - \int \mathbb{I}_k \log \Big(\frac{\mu_t^k}{\bar{\mu}_t^{\otimes k}}\Big).
\end{align*}
For the last term, by Proposition \ref{prop:5.1}, it holds
\begin{equation*}
    \E \int \frac{|\mu_{t,n}^{N,\xi_n} - \mu_{t,n}^{N,\xi_n^i}|^4}{(\mu_{t,n}^{N,\xi_n} + \mu_{t,n}^{N,\xi_n^i})^3} \le C\tau^4 \e^{Ct}.
\end{equation*}
Hence, combining the boundedness assumption of $b$, one has 
\begin{align}
 \int& \mathbb{I}_k \log \Big(\frac{\mu_t^k}{\bar{\mu}_t^{\otimes k}}\Big)=
         -\sum\limits_{i=1}^k \int \E \bigg[\sum_{\substack{\ell\in\xi_n(i)\\ \ell \ne i}}\frac{b(x_i - x_\ell)}{p-1} (\mu_{t,n}^{N,\xi_n} - \mu_{t,n}^{N,\xi_n^i})\bigg]\cdot\nabla_{x_i}\log \Big(\frac{\mu_t^k}{\bar{\mu}_t^{\otimes k}}\Big)\notag\\
    \le& \sum\limits_{i=1}^k \epsilon_3 \int \Big| \nabla_{x_i}\log \Big(\frac{\mu_t^k}{\bar{\mu}_t^{\otimes k}}\Big) \Big|^2 \mu_t^{N} \notag\\
        &+ \frac{C}{\epsilon_3} \sum\limits_{i=1}^k \int \E \Big|\sum_{\substack{\ell\in\xi_n(i)\\ \ell \ne i}}\frac{b(x_i - x_\ell)}{p-1}\Big|^2 \frac{|\mu_{t,n}^{N,\xi_n} - \mu_{t,n}^{N,\xi_n^i}|^2}{\mu_{t,n}^{N,\xi_n} + \mu_{t,n}^{N,\xi_n^i}}\notag\\
    \le& \sum\limits_{i=1}^k \epsilon_3 \int \Big| \nabla_{x_i}\log \Big(\frac{\mu_t^k}{\bar{\mu}_t^{\otimes k}}\Big) \Big|^2 \mu_t^{N} \notag\\
        &+ \frac{C}{\epsilon_3} \sum\limits_{i=1}^k \left(\E \int\Big|\sum_{\substack{\ell\in\xi_n(i)\\ \ell \ne i}}\frac{b(x_i - x_\ell)}{p-1}\Big|^4(\mu_{t,n}^{N,\xi_n} + \mu_{t,n}^{N,\xi_n^i}) \right)^{\frac{1}{2}}\left(\E \int \frac{|\mu_{t,n}^{N,\xi_n} - \mu_{t,n}^{N,\xi_n^i}|^4}{(\mu_{t,n}^{N,\xi_n} + \mu_{t,n}^{N,\xi_n^i})^3}\right)^{\frac{1}{2}}\notag\\
   \le & \epsilon_3 I_t^k + \frac{C}{\epsilon_3}\e^{Ct}k\tau^2, \label{eq:estimateIk}
\end{align}
    by Young's inequality and the Cauchy-Schwarz inequality.
Then, one obtains
\begin{equation}\label{eq:dHtk}
\begin{aligned}
    \frac{d}{dt}H_t^k\le& (-\sigma +\epsilon_1 + \epsilon_2 +\epsilon_3)I_t^k + \frac{C}{\epsilon_1 N^2}\sum\limits_{i=1}^k \int \mu_t^k \bigg|\sum\limits_{j=1,\,j\ne i}^k b(x_i - x_j) - (k-1)b*\bar{\mu}_t (x_i)\bigg|^2\\
    &+ \frac{C}{\epsilon_2}\frac{(N-k)k}{N}(H_t^{k+1}-H_t^k) + \frac{C}{\epsilon_3}\e^{Ct}k\tau^2.
\end{aligned}
\end{equation}
Taking $\epsilon_i,$ $i=1,2,3$, small enough such that $\sum\limits_{i=1}^3\epsilon_i<\sigma$, then
\begin{equation}\label{eq:dHtk2}
    \frac{d}{dt}H_t^k \le \frac{C}{N^2}\sum\limits_{i=1}^k \int \mu_t^k \bigg|\sum\limits_{j=1,\,j\ne i}^k b(x_i - x_j) - (k-1)b*\bar{\mu}_t (x_i)\bigg|^2 + C \e^{Ct}k\tau^2 + C k (H_t^{k+1}-H_t^{k}).
\end{equation}

\subsection{Boundedness of $H_t^N$}\label{subsec:boundHtN}

Before dealing with the iteration inequality \eqref{eq:dHtk2}, we first establish a coarse bound for $H_t^N$ by the Large Derivation Principle (details are provided in Section \ref{subsec: aux ineq}).
\begin{lemma}\label{lem:HN}
    Under the assumptions of Theorem \ref{thm:1}, it holds
    \begin{equation}\label{eq:HN}
        H_t^N \le C\e^{Ct} + C\e^{Ct}N\tau^2,
    \end{equation}
    where $C$ is independent of $N$.
\end{lemma}
\begin{proof}
    Recalling \eqref{eq:A-B}, note that $\mathcal{K}_N=0$. We only need to consider $\mathcal{J}_N$ in \eqref{eq:Jestimate1}. By H\"older's inequality, one has
    \begin{equation*}
        \mathcal{J}_N \le \epsilon I_t^N+ \sum\limits_{i=1}^N \int \mu_t^N \Big|\sum\limits_{j=1,\,j\ne i}^N \frac{1}{N-1}b(x_i - x_j) - b*\bar{\mu}_t(x_i)\Big|^2.
    \end{equation*}
    Define $\psi^i(x_j):= b(x_i-x_j)-b*\bar{\mu}_t(x_i).$ Then the second term of the above equality becomes
    \begin{equation*}
        \sum\limits_{i=1}^N \int \mu_t^N \Big|\sum\limits_{j=1,\,j\ne i}^N \frac{1}{N-1}b(x_i - x_j) - b*\bar{\mu}_t(x_i)\Big|^2 = \frac{1}{N-1}\sum\limits_{i=1}^N \int \mu_t^N \frac{1}{N-1}\Big|\sum\limits_{j=1,j\ne i}^N\psi^i(x_{j})\Big|^2,
    \end{equation*}
    which is finite due to the Lipschitz continuity of $b$ and Lemma \ref{lem:momentcontrol}.
    Then, applying Lemma \ref{lmm:changemeasure}, one has
    \begin{equation*}
        \int \mu_t^N \frac{1}{N-1}\Big|\sum\limits_{j=1,j\ne i}^N\psi^i(x_j)\Big|^2
        \le \frac{1}{\eta} \left( H_t^N + \log \int \exp\bigg\{ \frac{\eta}{N-1} \Big|\sum\limits_{j=1,j\ne i}^N\psi^i(x_j)\Big|^2 \bigg\} \bar{\mu}_t^{\otimes N}\right),
    \end{equation*}
    for any $\eta>0$. Note that 
    \begin{equation*}
        \int \psi^i(x_j)\bar{\mu}_t^{\otimes N} =0, \text{ for any } j\ne i.
    \end{equation*}
    Since the interaction kernel is globally Lipschitz, Lemma \ref{lem:sub-Gaussian} implies that the conditions in Lemma \ref{lem:LDP} hold for certain positive $\eta$ small enough.
    Then, by Lemma \ref{lem:LDP}, one has that, 
    \begin{equation*}
        \int \exp\bigg\{ \frac{\eta}{N-1} \Big|\sum\limits_{j=1,j\ne i}^N\psi^i(x_j)\Big|^2 \bigg\} \bar{\mu}_t^{\otimes N} \le C.
    \end{equation*}
    Hence, combining with the symmetry of the particle system, one obtains that
    \begin{equation*}
        \frac{d}{dt}H_t^N \le C H_t^N + C + C\e^{Ct}N\tau^2.
    \end{equation*}
    The Gr\"onwall inequality gives the result.
\end{proof}
Based on Lemma \ref{lem:HN}, as a direct corollary of linear scaling for relative entropy Lemma \ref{lem:linear}, we have
    \begin{equation}\label{eq:cor:Htk}
        H_t^k \le C\e^{Ct}\frac{k}{N} + C\e^{Ct}k\tau^2.
    \end{equation}

\subsection{Improvement of the result}\label{subsec: improve}

Now we improve the result by using the bounds of $H_t^k$, $1\le k\le N$.
By the similar analysis of Lemma \ref{lem:HN}, one has
\begin{equation}\label{eq:4.3.1}
    \int \mu_t^k \Big|\sum\limits_{j=1,\,j\ne i}^k b(x_i - x_j) - (k-1)b*\bar{\mu}_t(x_i)\Big|^2
    \lesssim k(H_t^k + C).
\end{equation}
Therefore, combining \eqref{eq:4.3.1} with \eqref{eq:cor:Htk}, one obtains
\begin{align*}
    \frac{1}{N^2}\sum\limits_{i=1}^k\int \mu_t^k \bigg|\sum\limits_{j=1,\,j\ne i}^k b(x_i - x_j)-(k-1)\bar{\mu}_t*b(x_i)\bigg|^2& \le C\e^{Ct}\frac{k^2}{N^2}(\frac{k}{N} + k\tau^2 +C)\\
    &\le C\e^{Ct}(\frac{k^2}{N^2}+ k\tau^2).
\end{align*}
Then \eqref{eq:dHtk2} derives
\begin{equation}\label{eq:dHtk3  0}
    \frac{d}{dt}H_t^k \le C \e^{Ct}\frac{k^2}{N^2}+ C \e^{Ct}k\tau^2 + C k (H_t^{k+1}-H_t^{k}).
\end{equation}
For noatational convenience, we rewrite \eqref{eq:dHtk3  0} into
\begin{equation}\label{eq:dHtk3}
    \frac{d}{dt}H_t^k \le C \e^{Ct}\frac{k^2}{N^2}+ C \e^{Ct}k\tau^2 + \gamma k (H_t^{k+1}-H_t^{k}),
\end{equation}
where $\gamma$ equals the constant $C$ in \eqref{eq:dHtk3  0}.
By Gr\"onwall's inequality,
\begin{equation*}
    H_t^k \le \e^{-\gamma k t}H_0^k + \int_0^t \e^{-\gamma k (t-s)}\left(C \e^{Ct}\frac{k^2}{N^2} + C \e^{Ct}k\tau^2 + \gamma k H_s^{k+1} \right)\,ds.
\end{equation*}
Note that $H_0^\ell = 0$ since the initial data are $i.i.d.$ Iterating this inequality $N-k$ times, one gets
\begin{equation}\label{eq:HtkAB}
    H_t^k \le \sum\limits_{\ell=k}^{N-1}  C \e^{Ct}(\frac{\ell}{N^2}+ \tau^2) A_k^\ell(t) + A_k^{N-1}(t)H_t^N,
\end{equation}
where $A_k^\ell$ is defined in Lemma \ref{lem:AB}.

By Lemma \ref{lem:AB}, one knows that 
\begin{equation*}
    \begin{aligned}
        A_k^{N-1}(t)\le & \exp\left\{ -2N \Big(\e^{-\gamma t} -\frac{k}{N}\Big)_+\right\}\\
    \le&\left\{ \begin{aligned}
        &\exp\left(- N \e^{\gamma t}\right), \quad &\text{if } k\le \frac{N\e^{-\gamma t}}{2},\\
    &1, &\text{otherwise}.
    \end{aligned}\right.
    \end{aligned}
\end{equation*}
For the first case of $k\le \frac{N\e^{-\gamma t}}{2}$, note that the elementary inequality
\begin{equation*}
    \e^{-a} \le \beta ! a^{-\beta},\quad \text{for } a>0 \text{ and } \beta \in \mathbb{Z}_+.
\end{equation*}
Taking $\beta =2$, one has
\begin{equation}\label{eq:ANcase1}
    \exp\left(- N \e^{\gamma t}\right) \le \left(N\e^{-\gamma t}\right)^{-2}2!\le \frac{C\e^{Ct}}{N^2}.
\end{equation}
For the second case, one has that
\begin{equation}\label{eq:ANcase2}
    k> \frac{N\e^{-\gamma t}}{2} \quad\Longrightarrow\quad 1< 4 \frac{k^2}{N^2}e^{2\gamma t}.
\end{equation}
Combining \eqref{eq:ANcase1} and \eqref{eq:ANcase2}, one obtains that, for any $k\ge 1$,
\begin{equation*}
    A_k^{N-1} (t)\le C\e^{Ct}\frac{k^2}{N^2}.
\end{equation*}
Therefore, by Lemma \ref{lem:HN}, it holds
\begin{equation*}
    A_k^{N-1} (t) H_t^N \le C\e^{Ct}\frac{k^2}{N^2}(C\e^{Ct} + C\e^{Ct}N\tau^2) \le C\e^{Ct}\frac{k^2}{N^2} + C\e^{Ct}k\tau^2.
\end{equation*}
The ODE hierarchies of $H_t^k$ now becomes
\begin{equation*}
    H_t^k \le \sum\limits_{\ell=k}^{N-1} C\e^{Ct}(\frac{\ell}{N^2}+ \tau^2) A_k^\ell(t) + A_k^{N-1}(t)H_t^N.
\end{equation*}
Using Lemma \ref{lem:AB} again, we have
\begin{equation*}
    H(\mu_t^k\,|\, \bar{\mu}_t^{\otimes k}) \le Ct\e^{Ct}\left(\frac{k^2}{N^2} + k\tau^2\right).
\end{equation*}
Now  the proof is finished.



\section{Auxiliary lemmas}\label{sec:auxlem}

In this section, we present the auxiliary lemmas applied in the main text. We begin by analyzing the solutions to the two-particle systems \eqref{eq:rbmalgo} and \eqref{eq:mckean} using the standard It\^o calculus. Then, in Section \ref{subsec:aux density}, we estimate the Fisher information of the density for the RBM particle system, which provides a crucial estimate for Proposition \ref{prop:5.1}. Sections \ref{subsec: aux ineq} and \ref{subsec: aux ode} contain technical lemmas used in the proof of Theorem \ref{thm:1}.

\subsection{Moment bounds}

Here, we establish coarse moment bounds for $\tilde{X}_i^\xi$ with $1\le i\le N$ and $q\ge 2$, which is essential for controlling the growth of the interaction term in the main proof. For the McKean SDE \eqref{eq:mckean}, Lemma \ref{lem:sub-Gaussian} shows that the solution is sub-Gaussian under the sub-Gaussian assumption of $\mu_0$,.

\begin{lemma}[moments control]\label{lem:momentcontrol}
    Suppose the assumptions of Theorem \ref{thm:1} holds. Then, for any $q\ge 2$, there exists a constant $C=C(d,\,q,\,\sigma,\,\|\nabla b\|_\infty,\,b_0)$ such that for any $i$,
    \begin{equation}\label{eq:momentcontrol}
         \E |\tilde{X}_i^{\xi}|^q \le C\e^{Ct}.
    \end{equation}
\end{lemma}
\begin{proof}
We can use a similar method as in \cite{jin2020random} to prove Lemma \eqref{eq:momentcontrol}. Under the assumptions of Theroem \ref{thm:1}, it's straightforward to compute the differentiation of $\E(|\tilde{X}_i^\xi|^q\mid \mathcal{F}_n)$ by It\^o's calculus and derive
\begin{equation*}
        \frac{d}{dt} \E [|\tilde{X}_i^{\xi}|^q] \le C q \E [|\tilde{X}_i^{\xi}|^q] + C,
\end{equation*}
by taking expectation about the randomness in $\mathcal{F}_n$. The Gr\"onwall's inequality leads to the result \eqref{eq:momentcontrol}. We omit the details but refer to \cite[Lemma 3.3]{jin2020random}.
\end{proof}

Another crucial observation is that the distribution $\bar{\mu}_t$ remains sub-Gaussian distribution for all $t\ge 0$.
\begin{lemma}\label{lem:sub-Gaussian}
    Under the assumptions of Theorem \ref{thm:1}, the following statements hold.
    \begin{enumerate}
        \item For any $t \in [0,T]$, the solution of the mean field McKean SDE \eqref{eq:mckean} is sub-Gaussian. 
        \item The interaction kernel $b(\cdot)$ and the marginal distribution $\bar{\mu}_t$ of the Fokker-Planck equation \eqref{eq:FP} satisfy: there exist $C>0$ such that $\forall x,y \in \mathbb{R}^d$ and $t \in [0,T]$,
\[|b(x-y) - b {*} \bar{\mu}_t(x)| \leq C(1 + |y|).
\]
    \end{enumerate}
\end{lemma}

\begin{proof}
    The first claim can be verified by calculating $\E \,\exp(c|\bar{X}|^2)$ via It\^o's formula. Note that $b*\bar{\mu}_t$ is uniformly Lipschitz for any $\bar{\mu}_t$. The second one is actually also obvious by the first-order moment bound for $\bar{X}(t)$.
\end{proof}

\subsection{Estimates of the density for the RBM system}\label{subsec:aux density}
To prove Proposition \ref{prop:5.1}, we start from the estimate of the generalized Fisher information in Lemma \ref{lem:5.1}, which is the adaption of \cite[Lemma 5.1]{du2024collision} to the random batch system.

\begin{lemma}\label{lem:5.1}
For any $i\in\{1,\cdots,N\}$, let $\xi_n$ be the random batch division at $t_n$. Under the assumptions of Theorem \ref{thm:1}, for $q\ge 2,$ define
\begin{equation*}
    \hat{I}_q^i (t) : = \int_{\mathbb{R}^{Nd}}|\nabla_{x_i}\log \mu_t^{N,\xi} |^q \mu_t^{N,\xi} dx.
\end{equation*}
Then for $t\in[t_n,t_{n+1}),$ it holds that
\begin{equation}\label{eq:lemlem5.1}
\sum\limits_{j\in\xi_n(i)}\hat{I}_q^j(t) \le \e^{C\left(t-t_n\right)}\sum\limits_{j\in\xi_n(i)}\hat{I}_q^j(t_n)+Cp\left(t-t_n\right),
\end{equation}
where $C=C(d,\,q,\,\sigma,\,\|\nabla b\|_\infty,\,\|\nabla^2 b\|_\infty)$. 
\end{lemma}

\begin{proof}
During the interval $t \in(t_n, t_{n+1}]$, the $N$ particles are divided into $N / p$ groups. The particles of any batch $\xi_n(i)$ interact with each other without affecting other particles. Hence, we can order the particles to be $\{\mathfrak{S}(1), \mathfrak{S}(2), \cdots, \mathfrak{S}(N)\}$ such that $\{\mathfrak{S}(pk-p+1), \cdots, \mathfrak{S}(pk)\}$ are in the same group.
For notional simplicity, we set $b(0)=0$ formally.
Introduce
\begin{equation*}
    x_k := 
    \begin{bmatrix}
        x_{\mathfrak{S}(pk-p+1)} \\
        \vdots \\
        x_{\mathfrak{S}(pk)}
    \end{bmatrix} \in \mathbb{R}^{p d},
\end{equation*}
and
\begin{equation*}
    \tilde{a}(x_k) := \sigma \, I_{pd \times pd}, \quad 
    \tilde{b}(x_k) := -\frac{1}{p-1}
    \begin{bmatrix}
        \sum\limits_{\ell \in \xi_n(\mathfrak{S}(pk-p+1))} b\left(x_{\mathfrak{S}(pk-p+1)} - x_\ell\right) \\
        \vdots \\
        \sum\limits_{\ell \in \xi_n(\mathfrak{S}(pk))} b\left(x_{\mathfrak{S}(pk)} - x_\ell\right)
    \end{bmatrix} \in \mathbb{R}^{p d}.
\end{equation*}

In addition, we define
\begin{equation*}
\begin{aligned}
& \mathcal{I}_k=\{\mathfrak{S}(pk-p+1),\cdots,\mathfrak{S}(pk)\}, \\
& h:=\log f, \quad f:=\mu_t^{N, \xi}, \\
& u^{k}:=\e^{h} \sum_{\alpha \in \mathcal{I}_{k}} \varphi(h_\alpha), \quad \varphi(x):=|x|^q.
\end{aligned}
\end{equation*}
The subsrcipts are defined as
\begin{equation*}
    \cdot _k := \nabla_{x_k}(\cdot),\quad \cdot_{,k}: =\nabla_{x_k}\cdot(\cdot), \quad
    \cdot_{kk}:=\nabla^2_{x_k}(\cdot),\quad \cdot_{,kk}:= \nabla^2_{x_k}:(\cdot),
\end{equation*}
where $\nabla^2_{x_k}$ denotes the Hessian matrix with respect to $x_k$. Furthermore, we denote
\begin{equation*}
    \tilde{a}^{k}:=\tilde{a}(x_k)=\sigma I_{pd\times pd}, \quad \tilde{a}^{k}:=\tilde{a}(x_k)=\sigma I_{pd\times pd}, \quad \tilde{a}_{,\,kk}^{k}:=\nabla_{x_k}^{2}: \tilde{a}(x_k)=0,
\end{equation*}
and
\begin{equation*}
    a^{k}:=\tilde{a}^{k}, \quad b^{k}:=\tilde{b}^{k}+\tilde{a}_{,\,k}^{k}=\tilde{b}^{k}, \quad c^{k}:=\tilde{b}_{,\,k}^{k}+\tilde{a}_{,\,kk}^{k}=\tilde{b}_{,\,k}^{k}.
\end{equation*}
Using the new notations, we have
\begin{align}
u_{t}^{k}&-\sum_{i=1}^{N/p}\left(\left(\tilde{a}^{i} u^{k}\right)_{,\,ii}+\left(\tilde{b}^{i} u^k\right)_{,\,i}\right)
=u_{t}^{k}-\sum_{i=1}^{N/p}\left(a^i: u_{i i}^{k}+b^i \cdot u_i^{k}+c^{i} u^{k}\right), \label{eq:Eu}\\
 =&\sum_{\alpha \in \mathcal{I}_k} \varphi(h_\alpha)\underbrace{\left(\partial_{t}\e^{h}-\sum_{i=1}^{N/p} a^i:\left(\e^{h}\right)_{i i}-\sum_{i=1}^{N/p} b^i\cdot\left(\e^{h}\right)_i-\sum_{i=1}^{N/p} c^{i} \e^{h}\right)}_{=:Eu1} \notag\\
& +\sum_{\alpha \in \mathcal{I}_{k}} \nabla \varphi(h_\alpha) \cdot \underbrace{\left(\partial_t h_\alpha-2 \sum_{i=1}^{N/p} a^i: (h_i \otimes h_{\alpha i})-\sum_{i=1}^{N/p} a^i h_{\alpha ii}-\sum_{i=1}^{N/p} b^i \cdot h_{\alpha i}\right) e^h}_{=:Eu2} \notag\\
& -\sum_{\alpha \in\mathcal{I}_k}\nabla^2 \varphi(h_\alpha): \left(\sum_{i=1}^{N/p} a^i: h_{i \alpha} \otimes h_{i \alpha}\right) \e^{h}. \notag
\end{align}
Here, the double dot product ``:" means the Frobenius inner product (for matrices) or a double tensor contraction (for higher-order tensors).
By the equation of $h$, it's clear that
$$ E_{u1}=0, $$
and
\begin{align}
 E_{u2}=&\underbrace{\left(\partial_{t} h-\sum_{i=1}^{N/p} a^i:\left(h_{ii}+h_i \otimes h_i\right)-\sum_{i=1}^{N/p} b^i h_i-\sum_{i=1}^{N/p} c^{i}\right)_\alpha}_{=0}  \notag\\
& +\sum_{i=1}^{N/p} a_\alpha^i: h_{ii}+\sum_{i=1}^{N/p} a_{2}^{i}: h_i \otimes h_i+\sum_{i=1}^{N/p} b_\alpha^i h_i+c_\alpha^{i} \notag\\
 =&\left(a_\alpha^k:\left(h_{kk}+h_k \otimes h_k\right)+b_\alpha^{k} \cdot h_k+c_\alpha^k\right) e^h \notag.
\end{align}
Then, \eqref{eq:Eu} becomes
\begin{align}
\sum_{\alpha \in \mathcal{I}_k}& \nabla \varphi(h_\alpha)\left(a_\alpha^{k}:\left(h_{kk}+h_k \otimes h_k\right)+b_\alpha^{k} \cdot h_k+c_\alpha^k\right) \e^{h} \notag\\
&-\sum_{\alpha\in\mathcal{I}_k} \nabla^{2} \varphi(h_\alpha):\left(\sum_{i=1}^{N/p} a^i:h_{\alpha i} \otimes h_{\alpha i}\right) \e^{h}\notag\\
=&\,q \e^{h} \sum_{\alpha \in\mathcal{I}_k} |h_\alpha|^{q-2}h_\alpha \cdot\left(b_\alpha^{k} \cdot h_k+c_\alpha\right)\notag\\
&-q \e^{h} \sum_{\alpha \in\mathcal{I}_k} |h_\alpha|^{q-2} \sum_{i=1}^{N/p} a^i :_i\left[\Big(I+(q-2)\frac{h_\alpha\otimes h_\alpha}{|h_\alpha|^2}:_\alpha\left(h_{\alpha i} \otimes h_{\alpha i}\right)\right] .\label{eq:Eu3}
\end{align}
The final term is a useful term that gives dissipation. In fact, one has
\begin{equation*}
    \begin{aligned}
        a^i :_i\left(I+(q-2)\frac{h_\alpha\otimes h_\alpha}{|h_\alpha|^2}:_\alpha\left(h_{\alpha i} \otimes h_{\alpha i}\right)\right)&\ge \sum\limits_\lambda \lambda a^i : \left[\big(e_\lambda \cdot \nabla_\alpha h_i\big)\otimes \big(e_\lambda \cdot \nabla_\alpha h_i\big)\right]\\
        &\ge \sum\limits_\lambda a^i : \left[\big(e_\lambda \cdot \nabla_\alpha h_i\big)\otimes \big(e_\lambda \cdot \nabla_\alpha h_i\big)\right]\\
        &\ge \sigma \sum\limits_{\beta \in \mathcal{I}_k}|h_{\alpha\beta}|^2,
    \end{aligned}
\end{equation*}
where $(\lambda, e_\lambda)$ are the eigenpairs of the matrix $I+(q-2)\frac{h_\alpha\otimes h_\alpha}{|h_\alpha|^2}$ and $|A|:=\sqrt{\sum\limits_{i,j}A_{ij}^2}$ is the Frobenius norm of the matrix $A$. Therefore, \eqref{eq:Eu3} gives
\begin{align}
     u_{t}^{k}&-\sum_{i=1}^{N/p}\left(\left(\tilde{a}^{i} u^{k}\right)_{,\,ii}+\left(\tilde{b}^{i} u^k\right)_{,\,i}\right)
     \le  q \e^{h} \sum_{\alpha \in\mathcal{I}_k} |h_\alpha|^{q-2} h_\alpha \cdot\left(b_\alpha^{k} \cdot h_k+c_\alpha\right)  \notag\\
    \le & \frac{C}{p-1}q\e^{h} \sum_{\alpha \in\mathcal{I}_k}|h_\alpha|^{q-1} \sum_{\beta \in\mathcal{I}_k}\left(\left|h_{\beta}\right|+1\right),\label{eq:Euend}
\end{align}
where the constant $C$ is related to $\sigma,\,\|\nabla b \|_{\infty},\, \|\nabla^2 b\|_\infty$. 
Integrating \eqref{eq:Euend} over the space, it holds that
$$
\frac{d}{d t} \int u_t^k \le  C \int \e^{h} \sum_{\alpha \in\mathcal{I}_k}\left(|h_\alpha|^{2}+1\right),
$$
since
$$
\sum_{\alpha,\beta \in\mathcal{I}_k}|h_\alpha|\left|h_{\beta}\right| 
\le \frac{1}{2} \sum_{\alpha,\beta \in\mathcal{I}_k}\left(|h_\alpha|^{2}+\left|h_{\beta}\right|^{2}\right) 
= p\sum\limits_{\alpha\in\mathcal{I}_k}|h_\alpha|^{2}.
$$
Then one has,
\begin{equation*}
    \frac{d}{dt}\sum\limits_{j\in\xi_n(i)}\hat{I}_q^i(t) \le C \sum\limits_{j\in\xi_n(i)}\hat{I}_q^i(t) + Cp,
\end{equation*}
where $C =C(d,\,q,\,\sigma,\,\|\nabla b \|_{\infty},\, \|\nabla^2 b\|_\infty)$.
By  Gr\"onwall's inequality, one has
\begin{equation*}
    \sum\limits_{j\in\xi_n(i)}\hat{I}_q^i(t) \le \e^{C(t-t_n)}\sum\limits_{j\in\xi_n(i)}\hat{I}_q^i(t_n) + \e^{C(t-t_n)}p(\e^{C(t-t_n)}-1). 
\end{equation*}
Since $t\in [t_n,t_{n+1})$ and $\tau \ll 1$, \eqref{eq:lemlem5.1} holds.
\end{proof}

As a direct consequence of the above results, we have the following lemma. 
\begin{lemma}\label{lem:Cor5.1}
For any $i\in\{1,\cdots,N\}$, let $\xi_n$ be the random batch division at $t_n$. Under the assumption of Lemma \ref{lem:5.1}, for any $t\in[t_n,t_{n+1}),$ it holds
\begin{equation}\label{eq:lemCor5.1}
\sum_{i=1}^{N} \int\left|\nabla_{x_i} \log \mu_{t,n}^{N,\xi_n}\right|^q \mu_{t,n}^{N,\xi_n} \le C  \e^{Ct} N,
\end{equation}
for any $q\ge 2$, with $C=C(d,\,q,\,\sigma,\,\|\nabla b\|_\infty,\,\|\nabla^2 b\|_\infty)$ independent of $N$ and $\xi$.
\end{lemma}
\begin{proof}
By Lemma \ref{lem:5.1}, taking the summation over $i$, one obtains
\begin{align*}
    \sum_{i=1}^{N} \int\left|\nabla_{x_i} \log \mu_{t,n}^{N,\xi_n}(t)\right|^q \mu_{t,n}^{N,\xi_n}(t) dx 
    \le& C p\frac{N}{p}\left(t-t_n\right)\\
    &+\e^{C\left(t-t_n\right)} \sum_{i=1}^{N} \int\left|\nabla_{x_i} \log \mu_{t,n}^{N,\xi_n}\left(t_n\right)\right|^q \mu_{t,n}^{N,\xi_n}\left(t_n\right) dx.
\end{align*}
Iterating the above equation gives
\begin{equation*}
    \begin{aligned}
& \sum_{i=1}^{N} \int\left|\nabla_{x_i} \log \mu_{t,n}^{N,\xi_n}(t)\right|^q \mu_{t,n}^{N,\xi_n}(t) \\
& \le \e^{Ct} \sum_{i=1}^{N} \int \left|\nabla_{x_i} \log \mu_0\right|^q \mu_0 d x+C N\left(t-t_n+\sum_{j=1}^n \e^{C\left(t-t_{j}\right)} \tau\right) \\
& \le C \e^{Ct} N.
\end{aligned}
\end{equation*}
Note that
\begin{equation*}
    \int\left|\nabla_{x_i} \log g\right|^q g=\int \frac{\left|\nabla_{x_i} g\right|^q}{g}.
\end{equation*}
Since the mapping $\left(g_1, g_2\right) \longmapsto \frac{\left|g_2\right|^q}{g_1}$ is convex, and
$$
\mu_{t,n}^{N, \xi_n}=\mathbb{E}\left[\mu_{t}^{N, \xi} \mid \xi_n\right],
$$
one has that
$$
\sum_{i=1}^{N} \int\left|\nabla_{x_i} \log \mu_{t,n}^{N \cdot \xi_n}\right|^q \mu_{t,n}^{N, \xi_n} \le C t \e^{Ct} N.
$$
\end{proof}

\subsection{Auxiliary inequalities}\label{subsec: aux ineq}
Here, we first present the weighted Csisz\'ar-Kullback-Pinsker inequality used in Section \ref{subsec:derivation}. Then we show three auxiliary lemmas used in Section \ref{subsec:boundHtN}: a Fenchel-Young-type inequality, the Large Derivation Theorem and the linear scaling for relative entropy.

Lemma \ref{lem:CKP} is a generalization of  Pinsker's inequality. This famous result belongs to Villani, with the complete proof appearing in \cite[Theorem 2.1]{villani2009optimal}. 

\begin{lemma}[\cite{villani2009optimal}, Theorem 2.1]\label{lem:CKP}
    Let $E$ be a measurable space, let $\mu, \nu$ be two probability measures on $E$, and let $\varphi$ be a nonnegative measurable function on $E$. Then it holds
    $$\|\varphi(\mu-\nu)\|_{TV} \leq \sqrt{2}\left(1+\log \int e^{\varphi(x)^2} d \nu(x)\right)^{1 / 2} \sqrt{H(\mu \mid \nu)},$$
    where the notation $\varphi(\mu-\nu)$ is a shorthand for the signed measure $\varphi \mu-\varphi \nu$.
\end{lemma}

Then, we show a Fenchel-Young-type inequality which enables us to obtain a measure exchange estimate. For its proof, sees e.g., \cite[Lemma 1]{jabin2018quantitative}.

\begin{lemma}[\cite{jabin2018quantitative}, Lemma 1]\label{lmm:changemeasure}
For any two probability measures $\rho$ and $\tilde{\rho}$ on a Polish space $E$ and some test function $F \in$ $L^1(\rho)$, one has that, for any $\eta>0$,
$$
\int_E F \rho(d x) \leq \frac{1}{\eta}\left(D_{KL}(\rho \| \tilde{\rho})+\log \int_E e^{\eta F} \tilde{\rho}(d x)\right) .
$$
\end{lemma}

Lemma \ref{lem:LDP} (\cite[Lemma 3.3]{du2024collision}) can be viewed as the Law of Large Numbers at exponential scale, generalizing the results in \cite[Theorem 3,4]{jabin2018quantitative}. For readers' convenience, here we briefly introduce the Hoeffding bound used in the below statement. The Hoeffding inequality \cite{vershynin2018high} claims that for $n$ independent centered real random variables $Y_1, \dots, Y_n$, there exists a universal constant $c_* > 0$ such that
\begin{equation}
    P\left(\left|\sum_{j=1}^n Y_j \right| \geq y \right) \leq 2\exp\left(-\frac{c_* y^2}{\sum_{j=1}^n \| Y_j \|_{\psi_2}^2} \right),\quad \forall\, y \geq 0,
\end{equation}
where the $\psi_2$ norm (or the Orlicz norm with $\psi_2(x) = \exp(x^2)-1$) for some sub-Gaussian random variable $X$ is given by
\begin{equation}
    \|X\|_{\psi_2} := \inf \left\{c > 0:\mathbb{E}\left[\exp(|x|^2/c^2)\right] \leq 2 \right\}.
\end{equation}

\begin{lemma}[\cite{du2024collision}, Lemma 3.3]\label{lem:LDP}
    Consider any $\rho$ being a probability measure of a sample space E. Suppose that $\psi(x)$ satisfies
    \begin{equation*}
        \int_E \psi(x) \rho =0
    \end{equation*}
    and for the universal constant $c_*>0$ in the Hoeffding's inequality, the following holds
    \begin{equation}
\|\psi(x)\|_\rho:=\inf \left\{c>0: \int_E \exp \left(|\psi(x)|^2 / c^2\right) \rho(d x) \mid \leq 2\right\}<c_*.
\end{equation}
Then,
\begin{equation}
\sup _{N \geq 1} \int_{E^N} \exp \left(\frac{1}{N}\left|\sum_{i=1}^N \psi\left(x_i\right)\right|^2\right) \rho^{\otimes N} \mathrm{dx}<\infty .
\end{equation}
\end{lemma}

The following linear scaling property of the relative entropy is well-known for controlling the marginal distribution.
\begin{lemma}[linear scaling for KL-divergence]\label{lem:linear}
Let $\nu^n \in \mathcal{P}_s(E^{n})$ be a symmetric distribution over some space tensorized space $E^n$ and $\bar{\nu} \in \mathcal{P}(E)$.
For $1 \leq k \leq n$, define its $k$-th marginal $\nu^{n:k}$ by
\begin{equation*}
    \nu^{n:k}(z_1,\dots,z_k) := \int_{E^{n-k}} \nu^N(z_1,\dots,z_n)dz_{k+1}\dots dz_n.
\end{equation*}
Assume that $\nu^{n:k}\ll \bar{\nu}^{\otimes k}$ for any $1\le k \le N.$
Then it holds that
\begin{equation*}
    H\left(\nu^{n:k} \mid \bar{\nu}^{\otimes k}\right) \leq 2\frac{k}{n}H\left(\nu^{n} \mid \bar{\nu}^{\otimes n}\right).
\end{equation*}
\end{lemma}
The proof can be found in  e.g. \cite[Lemma 3.9]{miclo2001genealogies}, \cite[Equation (2.10), page 772]{csiszar1984sanov}.

\subsection{Estimates of iterated exponential integrals}\label{subsec: aux ode}
Here we present a useful estimate attributed to Lacker \cite{lacker2023hierarchies,lacker2023sharp}. We omit the proof but refer to \cite[Section 5]{lacker2023hierarchies}.

\begin{lemma}[\cite{lacker2023hierarchies}, Lemma 4.8]\label{lem:AB}
For integers $\ell \ge k \ge 1$, with the constant $\gamma>0,$ define
\begin{equation*}
    A_k^\ell (t_k): = \Big(\Pi_{j=k}^\ell\gamma j\Big)\int_0^{t_k}\int_0^{t_{k+1}}\cdots\int_0^{t_\ell} \e^{-\sum\limits_{j=k}^\ell \gamma j(t_j-t_{j+1})}dt_{\ell+1}\cdots dt_{k+1}.
\end{equation*}
Then one has
\begin{equation}\label{eq:lemab1}
    A_k^{\ell}(t) \leq \exp \left(-2(\ell+1)\left(\e^{-\gamma t}-\frac{k}{\ell+1}\right)_{+}^2\right),
\end{equation}
where $x_+ :=\max\{0,x\}$.
Moreover, for integers $r \geq 0$,
\begin{equation}\label{eq:lemab2}
 \sum_{\ell=k}^{\infty} \ell^r A_k^{\ell}(t) \leq \frac{(k+r)!}{(k-1)!} \frac{\e^{\gamma(r+1) t}-1}{r+1}.
\end{equation}
\end{lemma}

\section*{Acknowledgment}
This work was partially supported by the National Key R\&D Program of China, Project Number 2021YFA1002800. The work of S. Jin and Y. Wang was financially supported by the National Science Foundation for International Senior Scientists grant No. 12350710181, and the Shanghai Municipal Science and Technology Key Project No. 22JC1402300.
The work of L. Li was partially supported by NSFC 12371400 and Shanghai Municipal Science and Technology Major Project 2021SHZDZX0102.

\bibliographystyle{plain}
\bibliography{main}

\appendix
\section{Computational details in Proposition \ref{prop:5.1}.}\label{app:prop5.1}
Here we present calculus details of \eqref{eq:ddt prop5.1} in the proof of Proposition \ref{prop:5.1}. 
By straightforward computation, one has
\begin{align*}
     \frac{d}{dt} \int \frac{|\delta f|^4}{(f+\tilde{f})^3} 
    =& 4\int \frac{(\delta f)^3}{(f+\tilde{f})^3} \left( -\nabla_x \cdot (b^{\xi_n} f) + \sigma \nabla_x^2 : f + \nabla_x \cdot (b^{\xi_n^i} \tilde{f}) - \sigma\nabla_x^2 : \tilde{f}\right)\\
    &- 3\int \frac{(\delta f)^4}{(f+\tilde{f})^4} \left( -\nabla_x \cdot (b^{\xi_n} f) + \sigma \nabla_x^2 : f - \nabla_x \cdot (b^{\xi_n^i} \tilde{f}) + \sigma\nabla_x^2 : \tilde{f}\right)\\
    =& -4\int \nabla_x\left(\frac{\delta f}{f+\tilde{f}}\right)^3 \left( \nabla_x [\sigma I \delta f] -b^{\xi_n} \delta f - \delta b^i \tilde{f} \right)\\
    & +3\int \nabla_x\left(\frac{\delta f)}{f+\tilde{f})}\right)^4 \left( \nabla_x  [\sigma I (f+\tilde{f})]- b^{\xi_n} (f+\tilde{f}) + \delta b^i \tilde{f}\right)\\
    =& -12 \int \left(\frac{\delta f}{f+ \tilde{f}}\right)^2 \nabla_x \left(\frac{\delta f}{f+\tilde{f}}\right) \cdot [\sigma I \nabla_x \delta f] \\
    & +6 \int \left(\frac{\delta f}{f+ \tilde{f}}\right)^2 \nabla_x \left(\frac{\delta f}{f+\tilde{f}}\right)^2 \cdot [\sigma I \nabla_x (f + \tilde{f})]\\
    & -12 \int \left(\frac{\delta f}{f+ \tilde{f}}\right)^2 \nabla_x \left(\frac{\delta f}{f+\tilde{f}}\right) \cdot (-b^{\xi_n} \delta f - \delta b^i \tilde{f})\\
    & +12 \int \left(\frac{\delta f}{f+ \tilde{f}}\right)^3 \nabla_x \left(\frac{\delta f}{f+\tilde{f}}\right) \cdot (-b^{\xi_n}(f + \tilde{f}) + \delta b^i \tilde{f})\\
    =&: I_1 + I_2 + I_3 + I_4.
\end{align*}
Note that for any $Nd\times Nd$ matrix $A$, the following identity holds
\begin{multline*}
    \nabla_x \left(\frac{\delta f}{f + \tilde{f}}\right) \cdot A \nabla
    _x \delta f - \frac{1}{2}\nabla_x \left( \frac{\delta f}{f + \tilde{f}}\right)^2 \cdot A \nabla_x (f+\tilde{f})
    \\= (f+\tilde{f}) A :  \left(\nabla_x\left(\frac{\delta f}{f + \tilde{f}}\right) \otimes \nabla_x \left(\frac{\delta f}{f + \tilde{f}}\right)\right)
\end{multline*}
Then, one has
\begin{equation*}
    \begin{aligned}
        I_1 + I_2 =& -12 \int \left(\frac{\delta f}{f + \tilde{f}}\right)^2 \left\{\nabla_x \left(\frac{\delta f}{f + \tilde{f}}\right) \cdot [\sigma I \nabla_x \delta f] - \frac{1}{2}\nabla_x \left(\frac{\delta f}{f + \tilde{f}}\right)^2 \cdot [\sigma I \nabla_x(f + \tilde{f})]\right\}\\
        =& -12 \int \left(\frac{\delta f}{f + \tilde{f}}\right)^2 (f+\tilde{f})\sigma I : \left(\nabla_x \left(\frac{\delta f}{f + \tilde{f}}\right) \otimes \nabla_x \left(\frac{\delta f}{f + \tilde{f}}\right) \right)\\
        \le & 0.
    \end{aligned}
\end{equation*}
The remaining two terms can be combined to equal
\begin{align*}
    I_3 +I_4 =& 12 \int \left(\frac{\delta f}{f + \tilde{f}}\right)^2 \nabla\left(\frac{\delta f}{f + \tilde{f}}\right) \cdot \left( b^{\xi_n} \delta f + \delta b^i \tilde{f} - b^{\xi_n} \delta f + \frac{\delta b^i \tilde{f} \delta f}{f + \tilde{f}}\right)\\
    =& 24 \int \left(\frac{\delta f}{f + \tilde{f}}\right)^2 \nabla_x \left(\frac{\delta f}{f + \tilde{f}}\right) \delta b^i \frac{f \tilde{f}}{f + \tilde{f}}\\
    =& 8 \int \nabla_x \left(\frac{\delta f}{f + \tilde{f}}\right)^3 \cdot \delta b^i \frac{f\tilde{f}}{f+\tilde{f}}\\
    =& - 8 \int \left(\frac{\delta f}{f + \tilde{f}}\right)^3 \nabla_x \cdot \left(\delta b^i \frac{f\tilde{f}}{f+\tilde{f}}\right).
\end{align*}
Therefore, one obtains \eqref{eq:ddt prop5.1}.

\end{document}